\documentclass[12pt,twoside,a4paper]{amsart}
\usepackage{amssymb}
\date{\today}

\def\nbh{neighborhood }

\def\ann{{\rm ann}}
\def\deg{\text{deg}\,}

\def\w{\wedge}

\def\dbar{\bar\partial}

\def\w{{\wedge}}

\def\D{{\mathcal D}}

\def\S{{\mathcal S}}

\def\CH{\mathcal{CH}}

\def\sHom{{\mathcal Hom\, }}
\def\codim{{\rm codim\,}}

\def\Ker{{\rm Ker\,  }}

\def\O{{\mathcal O}}

\def\L{{\mathcal L}}

\def\Re{{\rm Re\,  }}

\def\L{{\mathcal L}}
\def\U{{\mathcal U}}
\def\ann{{\rm ann\,}}

\def\Cu{{\mathcal C}}
\def\J{{\mathcal J}}

\def\be{\begin{equation}}
\def\ee{\end{equation}}

\newtheorem{thm}{Theorem}[section]
\newtheorem{lma}[thm]{Lemma}
\newtheorem{cor}[thm]{Corollary}
\newtheorem{prop}[thm]{Proposition}

\theoremstyle{definition}

\theoremstyle{remark}

\newtheorem{preremark}{Remark}
\newtheorem{preex}{Example}

\newenvironment{remark}{\begin{preremark}}{\qed\end{preremark}}
\newenvironment{ex}{\begin{preex}}{\qed\end{preex}}

\numberwithin{equation}{section}

\title[]{Uniqueness and factorization of  Coleff-Herrera currents}

\begin{document}

\date{\today}

\author{Mats Andersson}

\address{Department of Mathematics\\Chalmers University of Technology and the University of 
G\"oteborg\\S-412 96 G\"OTEBORG\\SWEDEN}

\email{matsa@math.chalmers.se}


\thanks{The author was
  partially supported by the Swedish Natural Science
  Research Council}

\begin{abstract}
We prove a uniqueness result   for
Coleff-Herrera currents which in particular means that if
$f=(f_1,\ldots, f_m)$ defines a complete intersection, then
the classical Coleff-Herrera product associated to $f$
is the unique  Coleff-Herrera current 
that is cohomologous to $1$ with respect to the operator $\delta_f-\dbar$,
where $\delta_f$ is  interior multiplication with $f$.
From  the uniqueness result we  deduce that any 
Coleff-Herrera current on a variety $Z$ 
is a finite sum of products of  residue currents with support on
$Z$ and  holomorphic forms.   
\end{abstract}




\maketitle

\section{Introduction}

Let $X$ be an $n$-dimensional complex manifold and let $Z$ be an analytic
 variety of pure codimension $p$.
The sheaf of Coleff-Herrera currents  (or currents of {\it residual type})
$\CH_Z$   consists of all $\dbar$-closed $(*,p)$-currents
$\mu$ with support on $Z$ such that $\bar \psi\mu=0$ for each $\psi$
vanishing on $Z$, and which in addition
fulfills   the so-called standard extension property, SEP, see below.
Locally, any $\mu\in\CH_Z$ can be realized as a meromorphic
differential operator acting on the current of integration
$[Z]$ (combined with contractions with   holomorphic vector fields),
see, e.g., \cite{Bj1} and \cite{Bj2}. 

\smallskip
The model case of a Coleff-Herrera current is the Coleff-Herrera product
associated to a complete intersection $f=(f_1,\ldots,f_p)$,
\begin{equation}\label{chprod}
\mu^f=\big[\dbar \frac{1}{f_1}\w\ldots\w \dbar\frac{1}{f_p}\big],
\end{equation}
introduced by  Coleff and Herrera in \cite{CH}.
Equivalent definitions are given in \cite{P} and \cite{PT};
see also \cite{HS}. 
It was proved in  \cite{DS} and \cite{P} that the annihilator 
of $\mu^f$ is equal to the ideal $\J(f)$ generated by $f$.
 Notice that formally \eqref{chprod} is just
the pullback under $f$ of the product 
$\mu^w=\dbar(1/w_1)\w\ldots\w\dbar(1/w_p)$.
One can also express $\mu^w$ as $\dbar$ of the Bochner-Martinelli form 
$$
B(w)=\sum_j (-1)^j\bar w_j 
d\bar w_1\w \ldots \w d\bar w_{j-1}\w d\bar w_{j+1}\w\ldots \w d\bar w_p
/|w|^{2p}.
$$
In \cite{PTY}, $f^*B$ is defined as a principal value current,
and it is  proved that $\mu_{BM}^f=\dbar f^*B$ is indeed equal
to $\mu^f$. However the proof is quite involved. An alternative
but still  quite  technical proof appeared in \cite{A1}. 
In this paper we prove a uniqueness result  which states that any
Coleff-Herrera current that is cohomologous to $1$ with respect to
the operator $\delta_f-\dbar$ (see Section~\ref{badhus} for definitions) must be equal
to $\mu^f$. In particular this implies that  $\mu^f=\mu_{BM}^f$.

\smallskip
It is well-known  that any Coleff-Herrera current can be written
$\alpha\w\mu^f$, where $\alpha$ is a holomorphic $(*,0)$-form and
$\mu^f$ is a Coleff-Herrera product for a complete intersection
$f$. However, unless $Z$ is a complete intersection itself 
the support of $\mu^f$ is  larger than $Z$.  
Using the uniqueness result we can prove

\begin{thm}\label{bass}
For any $\mu\in\CH_Z$ (locally) there are residue currents $R_I$ with support on $Z$ and
holomorphic $(*,0)$-forms  $\alpha_I$ such that
\begin{equation}\label{ko}
\mu=\sum_{|I|=p}' R_I\w \alpha_I.
\end{equation}
\end{thm}

Here $R_I$ are currents of 
 Bochner-Martinelli type from \cite{PTY} associated with a not  necessarily  
complete intersection. 
In particular, it follows that the Lelong current $[Z]$
admits a factorization  \eqref{ko}.

\smallskip

By the the uniqueness result  we also obtain  simple proofs of
the equivalence of various definitions of the SEP, as well as 
the equivalence of various conditions for the vanishing
of a Coleff-Herrera current.

\smallskip

We will  adopt the following definition of SEP:  {\it Given any holomorphic $h$ 
that does not vanish identically on any irreducible component of $Z$,
the function  $|h|^{2\lambda}\mu$,  a~priori defined
only for $\Re\lambda>>0$, has a current-valued analytic extension to $\Re\lambda>-\epsilon$,
and the value at $\lambda=0$ coincides with $\mu$.}
The reason  for this choice is merely practical;
for the equivalence to the classical definition, see Section~\ref{SEP}. 
Now, if $\mu\in\CH_Z$ has support on $Z\cap\{h=0\}$,  then
$|h|^{2\lambda}\mu$ must vanish if $\Re\lambda$ is large enough, and
by the uniqueness of analytic continuation thus $\mu=0$.
In particular, $\mu=0$ identically if $\mu=0$ on $Z_{reg}$.

\section{The Coleff-Herrera product}\label{bulle}

Let $f_1,\ldots, f_p$ define a complete intersection in $X$,
i.e., $\codim Z^f=p$, where $Z^f=\{f=0\}$.
Notice that \eqref{chprod} is elementarily defined if each
$f_j$ is a power of a coordinate function.
The general definition relies  on the possibility to resolve singularities:
By Hironaka's theorem we can locally find a resolution
$\pi\colon\tilde\U\to\U$ such that 
locally in $\tilde\U$, each $\pi^* f_j$ is a monomial times
a non-vanishing factor. It turns out  that  locally  $\mu^f$ is a 
sum of terms 
\begin{equation}\label{nt}
\sum_\ell \pi_*\tau_\ell
\end{equation}
where each $\tau_\ell$ is of the form
$$
\tau_\ell=\dbar \frac{1}{t_1^{a_1}}\w\ldots\w
\dbar\frac{1}{t_p^{a_p}}\w\frac{\alpha}{t_{p+1}^{a_{p+1}}\cdots t_r^{a_r}},
$$
$t$ is a suitable local coordinate system in $\tilde \U$, and $\alpha$ is a smooth function
with compact support. It is well-known that  $\mu^f$ is in $\CH_{Z^f}$
but for further reference we sketch a proof:
It follows immediately from the definition that $\mu^f$ is a 
$\dbar$-closed $(0,p)$-current
with support on $Z^f$.
Given any holomorphic function $\psi$ we may choose the
resolution so that also $\pi^*\psi$ is a monomial.
Notice that each $|\pi^*\psi|^{2\lambda}\tau_\ell$ has an analytic continuation
to $\lambda=0$ and that the value at $0$ is equal to $\tau_\ell$ if
none of  $t_1,\ldots, t_p$ is a factor in $\pi^*\psi$ and 
zero otherwise. According to this let us subdivide the set of $\tau_\ell$ into
two groups $\tau_\ell'$ and $\tau_\ell''$.
Notice that 
$|\psi|^{2\lambda}\mu^f=\sum_\ell\pi_*(|\pi^*\psi|^{2\lambda}\tau_\ell)$
admits an  analytic continutation and that the value at $\lambda=0$ 
is $\sum\pi_*\tau''_\ell$.
If  $\psi=0$ on $Z^f$, then 
$0=|\psi|^{2\lambda}\mu^f$,  and hence $\mu^f=\sum_\ell\pi_*\tau_\ell'$;
it now follows  that $\bar\psi\mu^f=d\bar\psi\w\mu^f=0$.
If $h$ is holomorphic and the zero set of $h$ intersects
$Z^f$ properly, then $T=\mu^f-|h|^{2\lambda}\mu^f|_{\lambda=0}$
is a current of the  type \eqref{nt} with  support
on $Y=Z^f\cap\{h=0\}$ that has codimension $p+1$. For the same reason as
 above, $d\bar \psi\w T=0$ for each holomorphic $\psi$ that vanishes
on $Y$ and by a standard argument it now follows that  $T=0$
for degree reasons. Thus $\mu^f$ has the SEP and so  $\mu^f\in\CH_{Z^f}$.
This  proof  is  inspired by a forthcoming joint paper, \cite{AW2},
with  Elizabeth Wulcan. 

\section{The uniqueness result}\label{badhus}

Let $f=(f_1,\ldots,f_m)$ be a holomorphic tuple on  some complex manifold $X$. 
It is  practical to introduce a (trivial) vector bundle $E\to X$ with
global frame $e_1,\ldots,e_m$  and consider $f=\sum f_j e_j^*$ as a section of the dual
bundle $E^*$, where $e_j^*$ is the dual frame.
Then $f$  induces a mapping $\delta_f$, interior multiplication with $f$, 
on the exterior algebra $\Lambda E$.  
Let $\Cu_{0,k}(\Lambda^\ell E)$ be  the sheaf   of  
$(0,k)$-currents  with values in $\Lambda^\ell E$, considered as 
as sections of the bundle $\Lambda(E\oplus T^*(X))$; thus 
a section of  $\Cu_{0,k}(\Lambda^\ell E)$ is given by  an  expression
$v=\sum'_{|I|=\ell} f_I\w e_I$
where $f_I$ are  $(0,k)$-currents  and  $d \bar z_j\w e_k
=-e_k\w d\bar z_j$ etc.
Notice that both $\dbar$ and $\delta_f$ act as anti-derivations on these spaces,
i.e.,
$\dbar (f\w g)=\dbar f\w g+(-1)^{\deg f} f\w\dbar g, $
if at least one of $f$ and $g$ is smooth, and similarly for $\delta_f$.
It is  straight forward to check that $\delta_f \dbar=-\dbar\delta_f$.
Therefore, if  $\L^k=\oplus_j \Cu_{0,j-k}(\Lambda^j E)$
and $\nabla_f=\delta_f-\dbar$,  then   $\nabla_f\colon\L^k\to\L^{k+1}$,
and $\nabla_f^2=0$. For example, $v\in\L^{-1}$ is of the form
$v=v_1+\cdots +v_m$, where $v_k$ is a $(0,k-1)$-current  with values in
$\Lambda^k E$. Also for a general current 
the subscript will denote degree in $\Lambda E$.

\begin{ex}[The Coleff-Herrera product]\label{chex}
Let $f=(f_1,\ldots, f_m)$ be a complete intersection in  $X$. 
The current 
\begin{multline}\label{vdef}
V=\Big[\frac{1}{f_1}\Big] e_1+\Big[\frac{1}{f_2}\dbar\frac{1}{f_1}\Big]
\wedge e_1\wedge e_2 + \\
\Big[\frac{1}{f_3}\dbar\frac{1}{f_2}\wedge\dbar\frac{1}{f_1}\Big]
\wedge e_1\wedge e_2\wedge e_3 +\cdots
\end{multline}
is in $\L^{-1}$ and solves $\nabla_f V=1-\mu^f\w e,$
where $\mu^f$ is the Coleff-Herrera product and
$e=e_1\wedge\ldots\wedge e_m$. For definition of the
coefficients of $V$ and the computational rules used here,
see \cite{P}; one can obtain a simple proof of these rules
by arguing as in Section~\ref{bulle}, see  \cite{AW2}.
\end{ex}

\begin{ex}[Residues of Bochner-Martinelli type]\label{bm}
Introduce a Hermitian metric on $E$ and let $\sigma$ be the section of $E$ over
$X\setminus Z^f$ with minimal pointwise norm such that $\delta_f\sigma=f\cdot\sigma=1$.
Then 
$$
u=\frac{\sigma}{\nabla_f\sigma}=\sigma+\sigma\w\dbar\sigma+\sigma\w(\dbar\sigma)^2+\cdots
$$
is smooth outside $Z^f$ and $\nabla_f u=1$ there.
It turns out, see \cite{A1}, that  $U$ has a natural current extension $U$ across $Z^f$, and 
if $p=\codim Z^f$, then 
$\nabla_f U=1-R^f,$ where
$
R^f=R^f_p+\cdots+ R^f_{m}.
$
Moreover, these currents have representations like \eqref{nt} so if  
$\xi\in\O(\Lambda^{m-p}E)$ and $\xi\w R^f_p$ is $\dbar$-closed, then
it is in $\CH_Z^f$ by the arguments  given in Section~\ref{bulle}.
Notice that
\begin{equation}\label{kanna}
R_k^f=\sum'_{|I|=k} R^f_I\w e_{I_1}\w\ldots\w e_{I_k}.
\end{equation}
If we choose the trivial metric, the coefficients $R^f_I$ are precisely the
currents introduced in \cite{PTY}. In particular, if $f$ is a complete intersection,
i.e., $m=p$, then
$R^f_{1,\ldots,p}=\mu^f_{BM}\w e.$
\end{ex}

\begin{thm}[Uniqueness for Coleff-Herrera currents]\label{unik}
Assume that $Z^f$ has pure codimension $p$. If $\tau\in\CH_{Z^f}$ and
there is a solution $V\in\L^{p-m-1}$ to $\nabla_f V=\tau\w e$, then
$\tau=0$.
\end{thm}

\begin{remark} If $Z^f$ does not have pure codimension, the  theorem
still holds (with the same proof)  with $\CH_{Z^f}$  replaced by
$\CH_{Z'}$, where $Z'$ is the irreducible components of $Z^f$ 
of maximal dimension.
\end{remark}

In view of Examples~\ref{chex} and \ref{bm} we get

\begin{cor}
Assume  that $f$ is a complete intersection.  If $\mu\in\CH_{Z^f}$ and
there is a current $U\in\L^{-1}$ such that 
$\nabla_f U=1-\mu\w e$, then $\mu$ is equal to the Coleff-Herrera product
$\mu^f$. In particular, $\mu_{BM}^f=\mu^f$.
\end{cor}

The proof of Theorem~\ref{unik}   relies  on the following lemma, which
is probably known. However, for the reader's convenience we 
include an outline of a proof.

\begin{lma}\label{stod} 
If  $\mu$ is in $\CH_Z$ and for each \nbh\  $\omega$ of $Z$ there is 
a current $V$ with support in $\omega$ such that $\dbar V=\mu$, then
$\mu=0$.
\end{lma}

\begin{proof}
Locally on $Z_{reg}$ we
can choose coordinates $(z,w)$ such that $Z=\{w=0\}$.  Since $\bar w_j\mu=0$
and $\dbar \mu=0$ it follows that $d\bar w_j\w\mu=0$, $j=1,\ldots,p$, and hence
$\mu=\mu_0d\bar w_1\w\ldots\w d\bar w_p$.  
From a Taylor expansion in $w$ we get  that 
\begin{equation}\label{galt}
\mu=\sum_{|\alpha|\le M-p}
a_\alpha(z)\dbar\frac{1}{w_1^{\alpha_1+1}}\w\ldots\w \dbar\frac{1}{w_p^{\alpha_p+1}},
\end{equation}
where $a_\alpha$ are the push-forwards of  
$\mu\w w^\alpha dw/(2\pi i)^p$  under the projection $(z,w)\mapsto z$.
Since $\mu$ is $\dbar$-closed it follows that $a_\alpha$ are holomorphic.
Notice that 
$$
\dbar\frac{1}{w_1^{\beta_1}}\w\ldots\w \dbar\frac{1}{w_p^{\beta_p}}
\w dw_1^{\beta_1}\w\ldots\w dw_p^{\beta_p}/(2\pi i)^p=\beta_1\cdots\beta_p[w=0],
$$
where $[w=0]$ denote the current of integration over $Z_{reg}$.
Now assume that $\dbar\gamma=\mu$ and $\gamma$ has support close to $Z$.
We have, for $|\beta|=M$,  that 
$$
\dbar(\gamma\w dw^\beta)=(2\pi i)^pa_{\beta-1}(z)\beta_1\cdots\beta_p[w=0].
$$
If $\nu$ is the component of $\gamma\w dw^\beta$ of bidegree $(p,p-1)$
in $w$, thus
$$
d_w\nu=\dbar_w\nu=(2\pi i)^pa_{\beta-1}\beta_1\cdots\beta_p[w=0].
$$
Integrating with respect to $w$ we get that $a_{\beta-1}(z)=0$. By finite 
induction we can conclude
that $\mu=0$. Thus $\mu$ vanishes on $Z_{reg}$ and by the SEP  it 
follows that $\mu=0$.
\end{proof}

\begin{proof}[Proof of Theorem~\ref{unik}]
Let $\omega$ be any \nbh of $Z$ and take a cutoff function $\chi$   that is
$1$ in a \nbh of $Z$ and with support in $\omega$. 
Let $u$ be any smooth solution to $\nabla_f u=1$ in $X\setminus Z^f$, cf., Example~\ref{bm}.
Then
$
g=\chi-\dbar\chi\w u
$
is a smooth form in $\L^0(\omega)$ and $\nabla_f g=0$. Moreover, the scalar term
$g_0$ is $1$ in a \nbh of $Z^f$. Therefore,
$$
\nabla_f[ g\w V]=g\w \tau\w e= g_0 \tau \w e=\tau\w e,
$$
and hence  the current coefficient $W$ of the top degree component of $g\w V$ 
is a solution to $\dbar W=\tau$ with support in $\omega$.
In view of Lemma~\ref{stod} we have that  $\tau=0$.
\end{proof}

\section{The factorization}

The double sheaf complex 
$\Cu_{0,k}(\Lambda^\ell E)$ is exact in the $k$ direction  except at $k=0$,
where  we have the cohomology $\O(\Lambda^\ell E)$. By a standard argument 
there are  natural isomorphisms 
\begin{equation}\label{iso}
\Ker_{\delta_f}\O(\Lambda^\ell E)/
\delta_f\O(\Lambda^{\ell+1})\simeq \Ker_{\nabla_f}\L^{-\ell}/ \nabla_f\L^{-\ell-1}.
\end{equation}
When $\ell=0$ the left hand side is $\O/\J(f)$,
where $\J(f)$ is the ideal sheaf generated by $f$.
We have the following  factorization result.

\begin{thm}\label{faktothm}
Assume that $Z^f$ has pure codimension $p$ and let $\mu\in\CH_{Z^f}$
be $(0,p)$ and 
such that $\J(f)\mu=0$. Then there is locally $\xi\in\O(\Lambda^{m-p}E)$
such that  
\begin{equation}\label{fakto}
\mu\w e=R_p^f\w\xi.
\end{equation}
\end{thm}

\begin{proof}
Since $\nabla_f(\mu\w e)=0$, by \eqref{iso}  there is $\xi\in\O(\Lambda^{m-p}E)$ such that 
$\nabla_f V=\xi-\mu\w e$. On the other hand, if $U$ is the current from
Example~\ref{bm}, then $\nabla_f (U\w\xi)=\xi-R^f\w\xi=\xi- R^f_p\w\xi$.
Now \eqref{fakto} follows from Theorem~\ref{unik}.
\end{proof}

\begin{proof}[Proof of Theorem~\ref{bass}] 
With no loss of generality we may assume that $\mu$ has bidegree $(0,p)$.
Let $g=(g_1,\ldots,g_m)$ be a tuple such that $Z^g=Z$. If $f_j=g_j^M$ and $M$ is large
enough, then $\J(f)\mu=0$ and hence by Theorem~\ref{faktothm} there is a form
$$
\xi=\sum'_{|J|=m-p}\xi_J\w e_J
$$
such that \eqref{fakto} holds. 
Then, cf., \eqref{kanna},   \eqref{ko} holds if 
$\alpha_I=\pm \xi_{I^c}$, where
$I^c=\{1,\ldots,m\}\setminus I$.
\end{proof}

\begin{ex} \label{bamse1}
Let  $[Z]$ be any variety of pure codimension and choose
$f$ such that $Z=Z^f$. It is not hard to prove  that 
(each term of) the Lelong current $[Z]$ is in $\CH_{Z}$, and hence
there is a holomorphic form $\xi$ such that $R^f_p\w\xi=[Z]\w e$.
(In fact, one can notice that the proof of Lemma~\ref{stod}
works for $\mu=[Z]$ just as well, and then one can
obtain \eqref{fakto} for $[Z]$ in the same way as for $\mu\in\CH_Z$. 
A~posteriori it follows that indeed $[Z]$ is in $\CH_Z$.)
There are natural ways to regularize the current $R_p^f$, see, e.g., 
\cite{HS},  and thus we get natural regularizations of
$[Z]$.
\end{ex}

Next we  recall the duality principle, \cite{DS}, \cite{P}:
If $f$ is a complete intersection, then 
\begin{equation}\label{dual}
\ann \mu^f=\J(f).
\end{equation}
In fact, if $\phi\in \ann \mu$, then $\nabla_f U\phi=\phi-\phi\mu\w e=\phi$
and hence $\phi\in \J(f)$ by \eqref{iso}.
Conversely, if $\phi\in \J(f)$, then there is a holomorphic $\psi$ such that
$\phi=\delta_f\psi=\nabla_f\psi$ and hence $\phi\mu=\nabla_f\psi\w \mu=
\nabla(\psi\w\mu)=0$.  

Notice that  $\sHom_{\O}(\O/\J(f),\CH_{Z^f}(\Lambda^pE))$ is  the sheaf of currents
$\mu\w e$ with $\mu\in\CH_{Z^f}$  that are annihilated by $\J(f)$.
From \eqref{dual} and Theorem~\ref{faktothm} we now get

\begin{thm}\label{oslo}
If $f$ is a complete intersection, then the sheaf mapping
\begin{equation}
\O/\J(f) \to \sHom_{\O}(\O/\J(f),\CH_Z(\Lambda^pE)),\quad   \phi\mapsto \phi\mu^f\w e,
\end{equation}
is an isomorphism.
\end{thm}

\section{The standard extension property}\label{SEP}

Given the other conditions in the definition of $\CH_Z$ the
SEP  is automatically fulfilled on $Z_{reg}$; this is
easily seen, e.g.,  as in the proof of Lemma~\ref{stod},
so the interesting case is when the zero set $Y$ of $h$ contains the singular locus of $Z$.
Classically the SEP is expressed as
\begin{equation}\label{sepclass}
\lim_{\epsilon\to 0}\chi(|h|/\epsilon)\mu=\mu,
\end{equation}
where   $Y\supset Z_{sing}$ and $h$ is not vanishing identically on any
irreducible component of $Z$.  Here $\chi(t)$  can be either the characteristic function
for the interval $[1,\infty)$ or some smooth approximand.

\begin{prop}
Let $\chi$ be a fixed function as above.
The class of $\dbar$-closed $(0,p)$-currents $\mu$ with support on $Z$ that are annihilated
by $\bar I_Z$ and satisfy \eqref{sepclass} coincides with our class $\CH_Z$.
\end{prop}

If  $\chi$ is not smooth the existence of the currents $\chi(|h|/\epsilon)\mu$ 
in a reasonable  sense for small $\epsilon>0$ is part of the statement.

\begin{proof}[Sketch of proof]
Let $f$ be a tuple such that  $Z=Z^f$. 
We first show that $R^f_p$  satisfies \eqref{sepclass}. 
From the arguments in Section~\ref{bulle}, cf.,  Example~\ref{bm},    we know  that
$R^f_p$  has a representation \eqref{nt} such that $\pi^*h$ is a pure monomial
(since the possible nonvanishing factor  can be incorporated in one of
the coordinates)
and none of the  factors in   $\pi^*h$ occurs
among the residue factors in  $\tau_\ell$. Therefore,  the existence of the product in
\eqref{sepclass} and the equality follow from the simple observation that 
\begin{equation}
\int_{s_1,\ldots,s_\mu}\chi(|s_1^{c_1}\cdots s^{c_\mu}_\mu|/\epsilon)\frac{\psi(s)}
{s_1^{\gamma_1}\cdots s_\mu^{\gamma_\mu}}\to
\int_{s_1,\ldots,s_\mu}
\frac{\psi(s)}
{s_1^{\gamma_1}\cdots s_\mu^{\gamma_\mu}}
\end{equation}
for test forms $\psi$, where the right hand side is a principal value integral.
Let temporarily $\CH_Z^{cl}$ denote the class of currents defined in the proposition.
Since each $\mu\in\CH_Z$ admits the representation \eqref{fakto} it follows
that $\mu\in \CH_Z^{cl}$. On the other hand, Lemma~\ref{stod} and
therefore Theorem~\ref{unik} and 
\eqref{fakto} hold for $\CH_Z^{cl}$ as well (with the same proofs), 
and thus we get the other inclusion.
\end{proof}

\section{Vanishing of  Coleff-Herrera currents}

We conclude with some equivalent condition for the vanishing of a
Coleff-Herrera current. This result   is proved by the ideas above, 
it should be well-known, but we have not
seen it in this way in the literature.

\begin{thm} Assume that $X$ is Stein and that the subvariety
$Z\subset X$ has pure codimension
$p$. If $\mu\in\CH_Z(X)$ and $\dbar V=\mu$ in $X$, then 
the  following are equivalent:

\noindent (i)  $\mu=0$.

\noindent (ii) For all $\psi\in\D_{n,n-p}(X)$ such that $\dbar\psi=0$ in some
\nbh of $Z$ we have that  
$$
\int V \w\dbar \psi=0.
$$

\noindent (iii) There is a solution to $\dbar w=V$ in $X\setminus Z$.

\noindent (iv) For each neighborhood $\omega$ of $Z$ there is a solution to
$\dbar w=V$ in  $X\setminus\bar \omega$.
\end{thm}

\begin{proof}
It is easy to check that (i) implies all the other conditions.
Assume that (ii) holds. 
Locally on $Z_{reg}=\{w=0\}$
we have \eqref{galt}, and by choosing
$\xi(z,w)=\psi(z)\chi(w)dw^\beta\w dz\w d\bar z$
for a suitable cutoff function $\chi$ and test functions $\psi$,
we can conclude from (ii) that $a_\beta=0$ if $|\beta|=M$.
By finite induction it follows that  $\mu=0$ there. Hence $\mu=0$ globally
by the SEP.
Clearly (iii) implies (iv). Finally, assume that (iv) holds. 
Given $\omega\supset Z$ choose
$\omega'\subset\subset\omega$  and a solution to
$\dbar w=V$ in $X\setminus \overline{\omega'}$. If we  extend $w$ arbitrarily across $\omega'$
the form $U=V-\dbar w$ is a solution to $\dbar U=\mu$ with  support in $\omega$.
In view of Lemma~\ref{stod} thus $\mu=0$.
\end{proof}

Notice that $V$ defines a Dolbeault cohomology class $\omega^\mu$ in 
$X\setminus Z$
that only depends on $\mu$, and that conditions (ii)-(iv) are statements
about this class.
For an interesting application,  fix a current $\mu\in\CH_Z$.
Then the theorem gives several equivalent ways to
express that a given $\phi\in\O$ belongs to the annihilator
ideal of $\mu$. In the case when $\mu=\mu^f$ 
for a complete intersection $f$, one gets back
the equivalent formulations of the duality
theorem from \cite{DS} and \cite{P}.

\begin{remark}
If $\mu$ is an arbitrary $(0,p)$-current with support on $Z$ and $\dbar V=\mu$
we get an analogous theorem if condition (i) is replaced by:
{\it $\mu=\dbar\gamma$ for some $\gamma$ with support on $Z$.} 
This follows from the Dickenstein-Sessa decomposition 
$\mu=\mu_{CH}+\dbar\gamma$, where $\mu_{CH}$ is in $\CH_Z$.
See \cite{DS} for the case $Z$ is a complete intersection and
\cite{Bj1} for the general case.
\end{remark}

\def\listing#1#2#3{{\sc #1}:\ {\it #2},\ #3.}


\begin{thebibliography}{9999}







\bibitem{A1}\listing{M.\ Andersson}
{Residue currents and ideals of holomorphic functions}
{Bull.\  Sci. Math., {\bf 128}, (2004), 481--512}






\bibitem{AW2}\listing{M.\ Andersson \& E.\ Wulcan}
{Decomposition of  residue currents}
{In preparation}{}{}{}{}


\bibitem{Bj1}\listing{J-E Bj\"ork}
{Residue calculus and $\D$-modules om complex manifolds}
{Preprint Stockholm (1996)}
{}{}{}
{}{}







\bibitem{Bj2}\listing{J-E Bj\"ork}
{Residues and $\mathcal D$-modules}
{The legacy of Niels Henrik Abel,  605--651, Springer, Berlin, 2004}
{}{}







\bibitem{CH}\listing{N.r.\ Coleff \& M.e.\ Herrera}
{Les courants r\'esiduels associ\'es \`a une forme m\'eromorphe}
{Lect. Notes in Math. {\bf 633},  Berlin-Heidelberg-New York (1978)}








\bibitem{DS}\listing{A.\ Dickenstein  \& C.\ Sessa}
{Canonical representatives in moderate cohomology}
{Invent. Math. {\bf 80} (1985),  417--434.}




\bibitem{P}\listing{M.\ Passare}
{Residues, currents, and their relation to ideals of holomorphic functions}
{Math.\  Scand.\   {\bf 62}  (1988),  75--152}









\bibitem{PT}\listing{M.\ Passare \& A.\ Tsikh}
{Residue integrals and their Mellin transforms}
{Canad. J. Math. {\bf 47}  (1995),  1037--1050}
{}{}


\bibitem{PTY}\listing{M.\ Passare \& A.\ Tsikh \&  A.\ Yger}
{Residue currents of the Bochner-Martinelli type}
{Publ.\ Mat.  {\bf 44} (2000), 85-117}

\bibitem{HS}\listing{H.\ Samuelsson}
{Regularizations of products of residue and principal value currents}
{J.\  Funct.\  Anal.  {\bf 239}  (2006),  566--593}







\end{thebibliography}
\end{document}